\documentclass[11pt,a4paper,reqno]{amsart}

\usepackage{graphicx} 
\usepackage{amsmath,amsthm,amssymb} 
\usepackage{here}
\usepackage[top=30truemm,bottom=30truemm,left=25truemm,right=25truemm]{geometry} 

\def\qed{\hfill $\Box$}
\newcommand{\vin}{\rotatebox{90}{$\in$}}

\newtheorem{defi}{Definition}
\newtheorem{thm}[defi]{Theorem}
\newtheorem{prop}[defi]{Proposition}
\newtheorem{lem}[defi]{Lemma}
\newtheorem{rem}[defi]{Remark}
\newtheorem{cor}[defi]{Corollary}

\begin{document}

\title[Totally umbilical submanifolds in space forms]
{Totally umbilical submanifolds in pseudo-Riemannian space forms}

\author[Y. Sato]{Yuichiro Sato}

\address{Department of Mathematical Sciences,
Tokyo Metropolitan University,
Minami-Osawa 1-1, Hachioji, Tokyo, 192-0397, Japan.}

\email{y-sato@tmu.ac.jp}



\begin{abstract}
A totally umbilical submanifold in pseudo-Riemannian manifolds is a fundamental notion, which is characterized by the condition that the second fundamental form is proportional to the metric. 
It is also a generalization of the notion of a totally geodesic submanifold. 
In this paper, we classify congruent classes of full totally umbilical submanifolds in non-flat pseudo-Riemannian space forms, and consider their moduli spaces. 
As a consequence, we show that some moduli spaces of isometric immersions between space forms which one of the same constant curvature are non-Hausdorff. 
\end{abstract}

\maketitle 

\section{Introduction} 
A totally umbilical submanifold in a (pseudo-)Riemannian manifold is a fundamental notion. 
For example, a complete non-totally geodesic, totally umbilical submanifolds in a Euclidean space is a round sphere. 
The definition of a totally umbilical submanifold is that the second fundamental form is  proportional to the metric on the submanifold. 
In the case of Riemannian geometry, there exist researches of totally umbilical submanifolds in various ambient spaces \cite{C0, CV, MT, OT, ST}. 

There exist researches which characterizes the totally umbilicity of submanifolds in pseudo-Riemannian space forms \cite{Aku, Ch, ChN, DaSiDaSi, YL}. 
Ahn, Kim and Kim \cite{AKK} gave a complete classification of totally umbilical submanifolds in pseudo-Euclidean spaces, which are flat space forms. 
For non-flat cases, there exist some recent researches \cite{C} by Chen. 
In \cite[Propositions~3.7, 3.8, Chapter~3]{C}, the following is mentioned: 
If $\phi : M^{m}_{s} \rightarrow \mathbb{S}^{n}_{p}(1)$ is a totally umbilical isometric immersion,  then it is congruent to an open portion of one of the following submanifolds: 
\begin{align*} 
&\bullet \mathbb{S}^{m}_{s}(r^2) \rightarrow \mathbb{S}^{m+1}_{s}(1) \ ; \ x \mapsto (x, \sqrt{1-r^2}) \ (0<r\leq1), \\
&\bullet \mathbb{S}^{m}_{s}(r^2) \rightarrow \mathbb{S}^{m+1}_{s+1}(1) \ ; \ x \mapsto (\sqrt{r^2-1}, x) \ (r\geq1), \\
&\bullet \mathbb{H}^{m}_{s}(-r^2) \rightarrow \mathbb{S}^{m+1}_{s}(1) \ ; \ x \mapsto (x, \sqrt{1+r^2}) \ (r>0), \\
&\bullet \mathbb{E}^{m}_{s} \rightarrow \mathbb{S}^{m+2}_{s+1}(1) \ ; \ x \mapsto \left (r\langle x, x \rangle_{s}+rb-\dfrac{r}{4},rx ,\sqrt{1+br^{2}} ,r\langle x, x \rangle_{s}-rb+\dfrac{r}{4}\right) \ (r>0, br^2 \geq -1), \\
&\bullet \mathbb{E}^{m}_{s} \rightarrow \mathbb{S}^{m+2}_{s+2}(1) \ ; \ x \mapsto \left(r\langle x, x \rangle_{s}+rb-\dfrac{r}{4},\sqrt{br^{2}-1} ,rx, r\langle x, x \rangle_{s}+rb+\dfrac{r}{4}\right) \ (r>0, br^2 \geq 1). 
\end{align*} 

This classification of totally umbilical submanifolds in $\mathbb{S}^{n}_{p}(1)$ is insufficient. 
In fact, the following example is not contained in the above list 
\begin{equation} 
\psi : \mathbb{S}^{m}_{s}(1) \rightarrow \mathbb{S}^{m+2}_{s+1}(1) \ ; \ x \mapsto (1,x,1). \label{missing} 
\end{equation} 
When we compute the mean curvature vector field $H$ of $\psi$, we get 
\begin{equation*} 
H = (1,0,\cdots,0,1) \in \mathbb{E}^{m+3}_{s+1}. 
\end{equation*}
Thus, $H$ is a non-zero lightlike vector field. Hence $\psi$ is non-totally geodesic. 
Moreover, there are some observations for this example (see Section~4). 
It is obvious that the co-dimension is two and the co-index is one. 
The terminology is explained in detail in Section 2. 

Dajczer and Fornari in \cite{DF} showed that let
$\phi : \mathbb{S}^{m}_{s}(1) \rightarrow \mathbb{S}^{m+n}_{s}(1)$ be an isometric immersion, then 
$\phi$ is totally geodesic, where $m \geq 2$ and $1 \leq n \leq m-s-1$. 
In addition, Dajczer and Rodriguez in \cite{DR} showed the following rigidity theorem: 
Let $\phi : \mathbb{S}^{m}_{s}(1) \rightarrow \mathbb{S}^{m+2}_{s+1}(1)$ be an isometric immersion with $m-s \geq 4$. If the set of totally geodesic points does not disconnect $\mathbb{S}^{m}_{s}(1)$, then $\phi$ is congruent to an isometric immersion of the following type; 
\begin{equation}
\mathbb{S}^{m}_{s}(1) \ni x \mapsto (f(x), x, f(x)) \in \mathbb{S}^{m+2}_{s+1}(1), \label{rigid}
\end{equation}
where $f : \mathbb{S}^{m}_{s}(1) \rightarrow \mathbb{R}$ is a smooth function. 
For any $a \in \mathbb{R} \setminus \{0\}$, the following isometric immersions 
\begin{equation*} 
\psi_{a} : \mathbb{S}^{m}_{s}(1) \rightarrow \mathbb{S}^{m+2}_{s+1}(1) \ ; \ x \mapsto (a,x,a) \end{equation*} 
are congruent to the above $\psi = \psi_{1}$. 
Namely, co-dimension two totally umbilical immersions $\psi_{a}$ are in a special case that $f$ is a  non-zero constant function for the mapping (\ref{rigid}). 
Moreover, the set of totally geodesic points is empty since $\psi_{a}$ is not totally geodesic but totally umbilical. 

In this paper, we give a complete classification of totally umbilical submanifolds in pseudo-spheres or pseudo-hyperbolic spaces, that is, we classify congruent classes of full totally umbilical submanifolds. 
Moreover, we consider moduli spaces of totally umbilical submanifolds. 
As a consequence, we show that some moduli spaces of isometric immersions between space forms which are of the same constant curvature are non-Hausdorff. 

As applications, we obtain some totally umbilical lightlike submanifolds in non-flat pseudo-Riemannian space forms. 
A lightlike submanifolds in pseudo-Riemannian manifolds is defined by some kind of degeneracy of the induced metric. 
See references \cite{BD} and \cite{DJ} in detail. 
A classification of totally umbilical lightlike submanifolds in pseudo-Riemannian space forms is an open problem. 

At the end of this paper, we devote in Section 4.4 to the study of parallel isometric immersions. 
As a consequence, we see that the existence of marginally trapped parallel isometric immersions from an indefinite symmetric $R$-space into a pseudo-sphere or a pseudo-hyperbolic space. 
An isometric immersion between pseudo-Riemannian manifolds is marginally trapped if the mean curvature vector field is not zero at arbitrary point, but its norm vanishes identically. 
This never occurs in the Riemannian geometry. 

\section{Totally umbilical submanifolds in pseudo-Riemannian space forms}
Let $m$ and $n$ be positive integers, 
$M^{m}$ an $m$-dimensional manifold and 
$\bar{M}^{n}$ an $n$-dimensional pseudo-Riemannian manifold with index $p$. 
Here, when $L$ is a manifold, $L^{m}_{s}$ denotes an $m$-dimensional pseudo-Riemannian manifold with index $s$. 
The notation $\cong$ means the existence of an isometric isomorphism between pseudo-Riemannian manifolds. 
Hereinafter, we assume that a $C^{\infty}$-mapping $\phi : M \rightarrow \bar{M}$ is an immersion. 
Then, we call $\phi(M)$ an immersed submanifold in $\bar{M}$. 
In particular, when $\phi$ is injective, and $M$ is homeomorphic to the image 
$\phi(M)$ as a subspace of $\bar{M}$, $\phi(M)$ is said to be an \textit{embedded submanifold} in $\bar{M}$. 
In pseudo-Riemannian geometry, we remark that the induced metric is not always non-degenerate on $M$ even if $\phi$ is an immersion. 
When the induced metric is non-degenerate, we call $\phi(M)$ a 
\textit{non-degenerate submanifold}, or a \textit{pseudo-Riemannian submanifold} 
in $\bar{M}^{n}_{p}$. 

As another situation, let $M^{m}_{s}, \bar{M}^{n}_{p}$ be pseudo-Riemannian manifolds, and $g, \bar{g}$ denote pseudo-Riemannian metrics of $M, \bar{M}$, respectively. 
When $\phi$ is an isometric immersion from $M^{m}_{s}$ into $\bar{M}^{n}_{p}$, i.e. 
$\phi^{\ast}\bar{g} = g$, we also call $\phi(M)$ a non-degenerate submanifold in 
$\bar{M}^{n}_{p}$. 

We call $\phi$ an \textit{$r$-lightlike immersion} when the induced metric $g$ by $\phi$ is degenerate, and its radical distribution has constant rank $r > 0$, 
where a radical distribution of $M$ is defined by 
\begin{equation*} 
\textrm{Rad}TM = \{ X \in TM \mid g(X, Y) = 0, \ \textrm{for all} \ Y \in TM \}. 
\end{equation*} 
When $\phi$ is an $r$-lightlike immersion, we call $\phi(M)$ an $r$-lightlike submanifold in $\bar{M}^{n}_{p}$. 
Then, there exists a vector bundle $\textrm{tr}(TM)$ over $M$ such that we obtain a direct sum decomposition of vector bundles
\begin{equation} 
\phi^{\ast}T\bar{M} = TM \oplus \textrm{tr}(TM), \label{Gauss-type}
\end{equation}
where $\textrm{tr}(TM)$ is called a transversally vector bundle of $M$. 
Note that the above $\oplus$ does not express an orthogonal direct sum decomposition, and the uniqueness of a transversally vector bundle does not hold. 
This decomposition (\ref{Gauss-type}) induces a Gauss-type formula
\begin{equation*} 
\bar{\nabla}_{X}Y = \check{\nabla}_{X}Y + B(X,Y) 
\end{equation*}
for any $X, Y \in \Gamma(TM)$, and $\bar{\nabla}$ is a Levi--Civita connection of $\bar{M}$. 
We call $\check{\nabla}$ and $B$ an induced connection and a second fundamental form with respect to $\textrm{tr}(TM)$, respectively. 

We define an $n$-dimensional \textit{singular pseudo-Euclidean space} with the signature $(p, q, r)$ as 
\begin{equation*} 
\mathbb{E}^{p,q,r} := \left (\mathbb{R}^{n}, (\cdot , \cdot) = - \sum_{i=1}^{p} dx_{i}^{2} + \sum_{j=p+1}^{p+q} dx_{j}^{2} + \sum_{k=p+q+1}^{n} 0 dx_{k}^{2} \right ), 
\end{equation*} 
where $n=p+q+r$ and $(x_{1}, \cdots, x_{n})$ expresses the canonical coordinates on $\mathbb{R}^{n}$ \cite{St}. We use the following notations:
\begin{itemize} 
\item When $r=0$, $\mathbb{E}^{p,q,0}$ is called a \textit{pseudo-Euclidean space} and we denote  it by $\mathbb{E}^{n}_{p}$ and the metric by $\langle \cdot, \cdot \rangle_{p}$. 
\item When $r=0, p=1$, $\mathbb{E}^{1, n-1, 0} = \mathbb{E}^{n}_{1}$ is called a \textit{Minkowski $n$-space} and we denote it by $\mathbb{L}^{n}$.
\item When $p=r=0$, $\mathbb{E}^{0, n, 0} = \mathbb{E}^{n}_{0}$ is nothing but a Euclidean $n$-space $\mathbb{E}^{n}$. 
\end{itemize} 

We remark that $r \geq 1$ if and only if the metric $(\cdot , \cdot )$ is degenerate. 
In the context of isotropic geometry, the notation $\mathbb{E}^{0, n-1, 1}$ would be denoted by $\mathbb{I}^{n}$ \cite{Sa}. 

We define non-flat pseudo-Riemannian space forms with index $p$ as 
\begin{equation*} 
\mathbb{S}^{n}_{p}(r^2):=\left\{ x \in \mathbb{E}^{n+1}_{p} \mid \langle x, x \rangle_{p}=r^{2}\right\}, \quad 
\mathbb{H}^{n}_{p}(-r^2):=\left\{ x \in \mathbb{E}^{n+1}_{p+1} \mid \langle x, x \rangle_{p+1}=-r^{2}\right\},
\end{equation*} 
where $r>0$. 
We call $\mathbb{S}^{n}_{p}(r^2)$ and $\mathbb{H}^{n}_{p}(-r^2)$ an $n$-dimensional \textit{pseudo-sphere} and \textit{pseudo-hyperbolic space}, respectively. 
When $p=0$, $\mathbb{S}^{n}_{0}(1)$ and $\mathbb{H}^{n}_{0}(-1)\cap\{x_{1}>0\}$ are simply a standard sphere $\mathbb{S}^{n}(1)$ and a hyperbolic space $\mathbb{H}^{n}(-1)$, respectively. When $p=1$, $\mathbb{S}^{n}_{1}(1)$ and $\mathbb{H}^{n}_{1}(-1)$ are called a \textit{de~Sitter $n$-spacetime} and an \textit{anti-de~Sitter $n$-spacetime}, denoted by $d\mathbb{S}^{n}(1), Ad\mathbb{S}^{n}(-1)$. 
 
Let $M^{m}_{s}, \bar{M}^{n}_{p}$ be pseudo-Riemannian manifolds, and $g, \bar{g}$ denote pseudo-Riemannian metrics of $M, \bar{M}$, respectively. Let $\phi : M^{m}_{s} \rightarrow \bar{M}^{n}_{p}$ be an isometric immersion, and $h, H$ the second fundamental form and mean curvature vector field of $\phi$, respectively. 
$\phi$ is called \textit{totally geodesic} if $h$ identically vanishes. 
$\phi$ is called \textit{totally umbilical} if, for all $X, Y \in \Gamma(TM)$, it holds 
\begin{equation*} 
h(X,Y)=g(X,Y)H. 
\end{equation*}
$\phi$ is called \textit{minimal} if $H$ identically vanishes. As an easy observation, we see that $\phi$ is totally geodesic if and only if $\phi$ is totally umbilical and minimal. 
In addition to these notions, $\phi$ is called \textit{marginally trapped} if $H \neq 0$ and $\bar{g}(H,H)=0$. 
Finally, when an isometric immersion $\phi : M^{m}_{s} \rightarrow \bar{M}^{n}_{p}$ is totally geodesic, totally umbilical, minimal or marginally trapped, we call the image $\phi(M)$ a totally geodesic, totally umbilical, minimal or marginally trapped submanifold in $\bar{M}$, respectively. 

Let $\phi$ be an $r$-lightlike immersion. We call it \textit{totally geodesic} if $B=0$, and \textit{totally umbilical} if there exists $\mathcal{H} \in \textrm{tr}(TM)$ such that 
\begin{equation*} 
B(X,Y)=g(X,Y)\mathcal{H} 
\end{equation*}
for all $X, Y \in \Gamma(TM)$. 
These notions are independent of the choice of transversally vector bundles.  

Two isometric immersions $\phi_{1}$ and $\phi_{2}$ given by 
\begin{equation*}
\phi_{i} : M^{m}_{s} \rightarrow \bar{M}^{n}_{p} \ (i=1,2) 
\end{equation*}
are said to be \textit{congruent} if there exists an isometry $\Psi$ of $\bar{M}$ such that $\phi_{2} = \Psi \circ \phi_{1}$. 
The congruency defines an equivalence relation on the set consisting of submanifolds. 

For $\varepsilon = \pm 1,0$, we define for brevity
\begin{equation*}
  \mathbb{M}^{n}_{p}(\varepsilon) := 
  \begin{cases}
    \mathbb{E}^{n}_{p} & (\varepsilon = 0), \\
    \mathbb{S}^{n}_{p}(1) \subset \mathbb{E}^{n+1}_{p} & (\varepsilon = 1), \\
    \mathbb{H}^{n}_{p}(-1) \subset \mathbb{E}^{n+1}_{p+1} & (\varepsilon = -1).
  \end{cases}
\end{equation*}

The followings are well-known results in pseudo-Riemannian geometry. 

\begin{thm}[{\cite[Proposition~4, Chapter~1]{An}}] \label{tot_geod_in_Euc} \rm
Any non-degenerate affine subspace in a pseudo-Euclidean space $\mathbb{E}^{n}_{p}$ is a totally geodesic submanifold. Conversely, 
any connected non-degenerate totally geodesic submanifold in $\mathbb{E}^{n}_{p}$ is an open subset of a non-degenerate affine subspace. 
\end{thm}

\begin{rem} \rm
Let $\Pi_{s,t,r}^{m}$ be a \textit{canonical $r$-lightlike $m$-plane} in $\mathbb{E}^{n}_{p}$ with signature $(s,t,r)$, i.e. 
\begin{equation*} 
\Pi_{s,t,r}^{m} := \{(\underbrace{z_{1}, \cdots, z_{r}, x_{1}, \cdots, x_{s}, 0, \cdots, 0}_{p}, \underbrace{0, \cdots, 0, y_{1}, \cdots, y_{t}, z_{1}, \cdots, z_{r}}_{n-p}) \in \mathbb{E}^{n}_{p} \}. 
\end{equation*}
Then, $\Pi_{s,t,r}^{m}$ is a totally geodesic $r$-lightlike submanifold in $\mathbb{E}^{m+n}_{p}$ \cite{BD} and isometric to $\mathbb{E}^{s,t,r}$. 

From Theorem~\ref{tot_geod_in_Euc}, arbitrary subspaces $V \subset \mathbb{E}^{n}_{p}$ are congruent to non-degenerate subspaces $\mathbb{E}^{m}_{s}$, or degenerate subspaces $\Pi_{s,t,r}^{m}$ up to isometry of $\mathbb{E}^{n}_{p}$.
\end{rem}

We define an $n$-dimensional \textit{lightcone} with index $p$ in $\mathbb{E}^{n+1}_{p+1}$ as follows 
\begin{equation*} 
\Lambda^{n}_{p}:=\{ x \in \mathbb{E}^{n+1}_{p+1} \setminus \{0\} \ | \ \langle x,x\rangle_{p} = 0 \}.  
\end{equation*}
The lightcone $\Lambda^{n}_{p}$ is a totally umbilical $1$-lightlike hypersurface in $\mathbb{E}^{n+1}_{p+1}$, that is, the induced metric on $\Lambda^{n}_{p}$ is degenerate \cite{BD}. 

We recall totally geodesic submanifolds in non-flat pseudo-Riemannian space forms. We define a \textit{pseudo $m$-subsphere} of $\mathbb{S}^{n}_{p}(r^2)$ by 
\begin{equation*} 
\left\{(x_{1},\cdots,x_{s},0\cdots,0,x_{s+1},\cdots,x_{m+1}) \in \mathbb{S}^{n}_{p}(r^2) \right\} \cong \mathbb{S}^{m}_{s}(r^{2}). 
\end{equation*}
Analogously, we define a \textit{pseudo-hyperbolic $m$-subspace} of $\mathbb{H}^{n}_{p}(-r^2)$ by
\begin{equation*} 
\left\{(x_{1},\cdots,x_{s+1},0\cdots,0,x_{s+1},\cdots,x_{m+1}) \in \mathbb{H}^{n}_{p}(-r^2) \right\} \cong \mathbb{H}^{m}_{s}(-r^{2}). 
\end{equation*}
Then, the followings hold: 
\begin{thm}[{\cite[Proposition~3.3, 3.4, Chapter~3]{C}}] \rm
Up to isometry, an $m$-dimensional non-degenerate totally geodesic submanifold of an $n$-dimensional pseudo-sphere $\mathbb{S}^{n}_{p}(r^2)$ is an open portion of a pseudo $m$-subsphere. 
Up to isometry, an $m$-dimensional non-degenerate totally geodesic submanifold of an $n$-dimensional pseudo-hyperbolic space $\mathbb{H}^{n}_{p}(-r^2)$ is an open portion of a pseudo-hyperbolic $m$-subspace. 
\end{thm} 

Here, we refer a classification of totally umbilical submanifolds in pseudo-Euclidean spaces.
\begin{thm}[{\cite[Proposition~3.1]{AKK}}, {\cite[Theorem~1.4]{Ma}}] \label{tot_umb_Euc} \rm
Let $\phi : M^{m}_{s} \rightarrow \mathbb{E}^{n}_{p}$ be a totally umbilical isometric immersion and $H$ its mean curvature vector field. 
Then, the image is congruent to an open portion of one of the following submanifolds: 
\begin{itemize}
\item[(1)] a totally geodesic pseudo-Euclidean subspace $\mathbb{E}^{m}_{s} \subset \mathbb{E}^{n}_{p}$ ($H=0$); 
\item[(2)] a pseudo $m$-sphere $\mathbb{S}^{m}_{s}(r^2) \hookrightarrow \mathbb{E}^{m+1}_{s} \subset \mathbb{E}^{n}_{p}$ ($\langle H, H \rangle_{p} > 0$); 
\item[(3)] a pseudo-hyperbolic $m$-space $\mathbb{H}^{m}_{s}(-r^2) \hookrightarrow \mathbb{E}^{m+1}_{s+1} \subset \mathbb{E}^{n}_{p}$ ($\langle H, H \rangle_{p} < 0$); 
\item[(4)] a flat marginally trapped submanifold $\mathbb{U}^{m}_{s}$ defined by 
\begin{equation*} 
\mathbb{E}^{m}_{s} \rightarrow \mathbb{E}^{m+2}_{s+1} \subset \mathbb{E}^{n}_{p} \ ; \ x \mapsto \left(\langle x, x \rangle_{s} + \frac{1}{4}, x, \langle x, x \rangle_{s} - \frac{1}{4}\right) \ (H \neq 0, \ \langle H, H \rangle_{p} = 0). 
\end{equation*} 
\end{itemize}
\end{thm}

An isometric immersion $\phi : M^{m}_{s} \rightarrow \mathbb{M}^{n}_{p}(\varepsilon)$ is called \textit{full} if the image $\phi(M)$ is not contained in any non-degenerate totally geodesic hypersurface in $\mathbb{M}^{n}_{p}(\varepsilon)$.

\begin{lem}[Erbacher--Magid Reduction Theorem, {\cite[Theorem]{E}}, {\cite[Theorem~0.2]{Ma}}] \rm \label{reduction} 
Let $\phi : M^{m}_{s} \rightarrow \mathbb{E}^{n}_{p}$ be an isometric immersion. For each $x \in M^{m}_{s}$, we define 
\begin{equation*} 
N^{0}(x) := \{\xi \in T_{x}^{\bot}M \ | \ A_{\xi} = 0 \} 
\end{equation*}
and define a \textit{first normal space} as the orthogonal complement of $N^{0}(x)$, i.e. 
\begin{equation*} 
N^{1}(x) = (N^{0}(x))^{\bot}. 
\end{equation*}
If a normal subbundle 
$N^{1} = \bigcup_{x \in M} N^{1}(x) \subset T^{\bot}M$ is parallel with respect to the normal connection, then there exists a geodesically complete $(m+k)$-dimensional (possibly lightlike) totally geodesic submanifold $E^{\ast} \subset \mathbb{E}^{n}_{p}$ such that $\phi(M) \subset E^{\ast}$, where $k = \textrm{rank} N^{1}$. 
\end{lem}

\begin{rem} \rm 
Let $M^{m}$ be a geodesically complete $m$-dimensional (possibly lightlike) totally geodesic submanifold in a pseudo-Euclidean space $\mathbb{E}^{n}_{p}$. Then, up to translation, $M^{m}$ coincides with a subspace of $\mathbb{E}^{n}_{p}$ because of the completeness. 
This claim holds by the fact that any geodesic of $M^{m}$ is a geodesic of $\mathbb{E}^{n}_{p}$, i.e. a line segment. 
\end{rem} 

\begin{lem}[{\cite[Corollary~3.1, Chapter~3]{C}}] \rm \label{umb_umb}
Let $\phi : M^{m}_{s} \rightarrow \mathbb{S}^{n}_{p}(1)$ (resp. $\mathbb{H}^{n}_{p}(-1)$) be an isometric immersion, and $\iota : \mathbb{S}^{n}_{p}(1) \rightarrow \mathbb{E}^{n+1}_{p}$ (resp. $\mathbb{E}^{n+1}_{p+1}$) the canonical inclusion map. When we set a mapping $f = \iota \circ \phi$, the followings hold:
\begin{itemize}
\item[(1)] $\phi$ has parallel mean curvature vector if and only if $f$ has parallel mean curvature vector; 
\item[(2)] $\phi$ is parallel if and only if $f$ is parallel; 
\item[(3)] $\phi$ is totally umbilical if and only if $f$ is totally umbilical.
\end{itemize}
\end{lem}
See Section 4 in this paper for the definition of parallel isometric immersions. 
The followings are main results in this paper. 
\begin{thm} \rm \label{main1}
If $\phi : M^{m}_{s} \rightarrow \mathbb{S}^{n}_{p}(1)$ is a full totally umbilical isometric immersion, $\bar{g}$ is the metric of $\mathbb{S}^{n}_{p}(1)$, and $H$ is its mean curvature vector field, 
then, up to isometry, the image is congruent to an open portion of one of the followings: 
\begin{itemize}
\item[(1)] $\mathbb{S}^{m}_{s}(1) \rightarrow \mathbb{S}^{m+1}_{s}(1) \subset \mathbb{S}^{n}_{p}(1) \ ; \ x \mapsto (x,0)$ \ (totally geodesic, $H=0$); 
\item[(2)] $\mathbb{S}^{m}_{s}(1) \rightarrow \mathbb{S}^{m+1}_{s+1}(1) \subset \mathbb{S}^{n}_{p}(1) \ ; \  x \mapsto (0,x)$ \ (totally geodesic, $H=0$); 
\item[(3)] $\mathbb{S}^{m}_{s}\left(r^2\right) \rightarrow \mathbb{S}^{m+1}_{s}(1) \subset \mathbb{S}^{n}_{p}(1) \ ; \  x \mapsto (x, \sqrt{1-r^2})$ \ ($0<r<1, \ \bar{g}(H,H) > 0$); 
\item[(4)] $\mathbb{S}^{m}_{s}\left(r^2\right) \rightarrow \mathbb{S}^{m+1}_{s+1}(1) \subset \mathbb{S}^{n}_{p}(1) \ ; \  x \mapsto (\sqrt{r^2-1}, x)$ \ ($r>1, \ -1 < \bar{g}(H,H) < 0$); 
\item[(5)] $\mathbb{S}^{m}_{s}(1) \rightarrow \mathbb{S}^{m+2}_{s+1}(1) \subset \mathbb{S}^{n}_{p}(1) \ ; \  x \mapsto (1,x,1)$ \ ($H \neq 0, \ \bar{g}(H,H) = 0$) ; 
\item[(6)] $\mathbb{H}^{m}_{s}\left(-r^2\right) \rightarrow \mathbb{S}^{m+1}_{s+1}(1) \subset \mathbb{S}^{n}_{p}(1) \ ; \  x \mapsto (x, \sqrt{1+r^2})$ \ ($r>0, \ \bar{g}(H,H) < -1$); 
\item[(7)] $\mathbb{E}^{m}_{s} \rightarrow \mathbb{S}^{m+1}_{s+1}(1) \subset \mathbb{S}^{n}_{p}(1) \ ; \  x \mapsto \left(\langle x, x \rangle_{s}-\dfrac{3}{4},x,\langle x, x \rangle_{s}-\dfrac{5}{4}\right)$ \ ($\bar{g}(H,H) = -1$). 
\end{itemize}
Moreover, when $M^{m}_{s}$ is geodesically complete, the image globally coincides with one of the above list (1)--(7). 
\end{thm}

\begin{thm} \rm \label{main2}
If $\phi : M^{m}_{s} \rightarrow \mathbb{H}^{n}_{p}(-1)$ is a full totally umbilical isometric immersion, $\bar{g}$ is the metric of $\mathbb{S}^{n}_{p}(1)$, and $H$ is its mean curvature vector field, 
then, up to isometry, the image is congruent to an open portion of one of the followings: 
\begin{itemize}
\item[(1)] $\mathbb{H}^{m}_{s}(-1) \rightarrow \mathbb{H}^{m+1}_{s}(-1) \subset \mathbb{H}^{n}_{p}(-1) \ ; \ x \mapsto (x,0)$ \ (totally geodesic, $H=0$); 
\item[(2)] $\mathbb{H}^{m}_{s}(-1) \rightarrow \mathbb{H}^{m+1}_{s+1}(-1) \subset \mathbb{H}^{n}_{p}(-1) \ ; \  x \mapsto (0,x)$ \ (totally geodesic, $H=0$); 
\item[(3)] $\mathbb{H}^{m}_{s}\left(-r^2\right) \rightarrow \mathbb{H}^{m+1}_{s+1}(-1) \subset \mathbb{H}^{n}_{p}(-1) \ ; \  x \mapsto (\sqrt{1-r^2}, x)$ \ ($0<r<1, \ \bar{g}(H,H) < 0$); 
\item[(4)] $\mathbb{H}^{m}_{s}\left(-r^2\right) \rightarrow \mathbb{H}^{m+1}_{s}(-1) \subset \mathbb{H}^{n}_{p}(-1) \ ; \  x \mapsto (x, \sqrt{r^2-1})$ \ ($r>1, \ 0 < \bar{g}(H,H) < 1$); 
\item[(5)] $\mathbb{H}^{m}_{s}(-1) \rightarrow \mathbb{H}^{m+2}_{s+1}(-1) \subset \mathbb{H}^{n}_{p}(-1) \ ; \  x \mapsto (1,x,1)$ \ ($H \neq 0, \ \bar{g}(H,H) = 0$) ; 
\item[(6)] $\mathbb{S}^{m}_{s}\left(r^2\right) \rightarrow \mathbb{H}^{m+1}_{s}(-1) \subset \mathbb{H}^{n}_{p}(-1) \ ; \  x \mapsto (\sqrt{1+r^2}, x)$ \ ($r>0, \ \bar{g}(H,H) > 1$); 
\item[(7)] $\mathbb{E}^{m}_{s} \rightarrow \mathbb{H}^{m+1}_{s}(-1) \subset \mathbb{H}^{n}_{p}(-1) \ ; \  x \mapsto \left(\langle x, x \rangle_{s}+\dfrac{5}{4},x,\langle x, x \rangle_{s}+\dfrac{3}{4}\right)$ \ ($\bar{g}(H,H) = 1$). 
\end{itemize}
Moreover, when $M^{m}_{s}$ is geodesically complete, the image globally coincides with one of the above list (1)--(7). 
\end{thm}

\section{Proof of main results}
Since the argument is parallel, we only give a proof in the case of pseudo-spheres. 
We assume that 
$\phi : M^{m}_{s} \rightarrow \mathbb{S}^{n}_{p}(1)$ is a totally umbilical isometric immersion, and $\iota : \mathbb{S}^{n}_{p}(1) \hookrightarrow \mathbb{E}^{n+1}_{p}$ is the inclusion. 
Then, $f := \iota \circ \phi : M^{m}_{s} \rightarrow \mathbb{E}^{n+1}_{p}$ is totally umbilical because of Lemma~\ref{umb_umb} (3). 
When we set $\tilde{h}$ and $\tilde{H}$ as the second fundamental form and the mean curvature vector field of $f$, respectively, we have, for any $X, Y \in \Gamma(TM)$, 
\begin{equation*} 
\tilde{h}(X, Y) = \langle X, Y \rangle_{p} \tilde{H}. 
\end{equation*}
In other words, when we set $\tilde{A}$ as the shape operator of $f$, we have 
\begin{equation*} 
\tilde{A}_{\xi}(X) = \langle \xi,\tilde{H} \rangle_{p}X 
\end{equation*}
for any $\xi \in \Gamma(T^{\bot}M), X \in \Gamma(TM)$. 
Therefore, we have 
\begin{equation*}
N^{1} = \textrm{Span} \{ \tilde{H} \}.
\end{equation*} 
In addition, if $\phi$ is totally umbilical, then $\phi$ has parallel mean curvature from the Codazzi equation. 
By using Lemma~\ref{reduction}, 
there exists an $(m+1)$-dimensional complete totally geodesic submanifold $E^{\ast} \subset \mathbb{E}^{n+1}_{p}$ such that $f(M) \subset E^{\ast}$. 
We see that $\phi(M) \subset \mathbb{S}^{n}_{p}(1) \cap E^{\ast}$. 
The subspace $E^{\ast}$ is congruent to one of $\mathbb{E}^{m+1}_{s}, \mathbb{E}^{m+1}_{s+1}, \Pi^{m+1}_{s,m-s,1}$. 
Thus, it suffices to check these three cases of $E^{\ast}$. 
Cohesively we use the formal notation $\bar{s} \in \{s, s+1\}$. 
Then, we have only to check the two possibilities $\mathbb{E}^{m+1}_{\bar{s}}$ and $\Pi^{m+1}_{s,m-s,1}$. 

Here, if we should take translations of $E^{\ast}$ into account, 
the direction of translations has to be transverse to $E^{\ast}$. 
When $E^{\ast}$ is congruent to $\mathbb{E}^{m+1}_{\bar{s}}$, i.e. in the non-degenerate case, taking a vector $v \in (\mathbb{E}^{m+1}_{\bar{s}})^{\bot}$, we have only to consider $E^{\ast} = \mathbb{E}^{m+1}_{\bar{s}} + v$. 
When $E^{\ast}$ is congruent to $\Pi^{m+1}_{s,m-s,1}$, i.e. in the degenerate case, 
$(\Pi^{m+1}_{\bar{s},t,1})^{\bot}$ is no longer a complementary of $\Pi^{m+1}_{\bar{s},t,1}$. 
From the viewpoint of lightlike geometry \cite{BD}, 
by regarding $\Pi^{m+1}_{s,m-s,1}$ as a $1$-lightlike submanifold in $\mathbb{E}^{n}_{p}$, 
we have decompositions 
\begin{align*}
\Pi^{m+1}_{s,m-s,1} &= \textrm{Rad}\Pi^{m+1}_{s,m-s,1} \oplus \mathbb{E}^{m}_{s}, \\
(\Pi^{m+1}_{s,m-s,1})^{\bot} &= \textrm{tr}\Pi^{m+1}_{s,m-s,1} \oplus \mathbb{E}^{n-m}_{p-s-1}, \\
\mathbb{E}^{n+1}_{p} &= \Pi^{m+1}_{s,m-s,1} \oplus \mathbb{E}^{n-m}_{p-s-1} \oplus \textrm{tr}\Pi^{m+1}_{s,m-s,1}, 
\end{align*}
where we define 
\begin{align*}
\textrm{Rad}\Pi^{m+1}_{s,m-s,1} &:= \textrm{Span}_{\mathbb{R}}\{\xi:= (1, 0,  \cdots, 0, 1) \in \Pi^{m+1}_{s,m-s,1}\},\\ 
\textrm{tr}\Pi^{m+1}_{s,m-s,1} &:= \textrm{Span}_{\mathbb{R}}\left\{N:=\frac{1}{2} (-1, 0, \cdots, 0, 1) \in (\Pi^{m+1}_{s,m-s,1})^{\bot} \right\}
\end{align*}
and $\xi, N$ satisfy $\langle \xi, \xi \rangle_{p} = \langle N, N \rangle_{p}=0, \langle \xi, N \rangle_{p}=1$. In the case, 
taking a vector $v \in \mathbb{E}^{n-m}_{p-s-1} \oplus \textrm{tr}\Pi^{m+1}_{s,m-s,1}$, 
we have only to consider $E^{\ast} = \Pi^{m+1}_{s,m-s,1} + v$. 

Under the above preparation and by isometry of $\mathbb{S}^{n}_{p}(1)$, we consider sequences of subspaces of $\mathbb{E}^{n+1}_{p}$;  
\begin{equation}
  \begin{array}{ccccccc}
     {\empty} & {\empty} & \mathbb{E}^{m+1}_{\bar{s}}+v_{S} & {\subset} & \mathbb{E}^{m+2}_{\bar{s}} & \rotatebox{-30}{$\subset$} & {\empty} \\
     {\empty} & \rotatebox{30}{$=$} & {\empty} & {\empty} & {\empty} & {\empty} & {\empty} \\ 
     E^{\ast} & {=} & \mathbb{E}^{m+1}_{\bar{s}}+v_{T} & {\subset} & \mathbb{E}^{m+2}_{\bar{s}+1} & {\subset} & \mathbb{E}^{n+1}_{p}, \\
     {\empty} & \rotatebox{-30}{$=$} & {\empty} & {\empty} & {\empty} & {\empty} & {\empty} \\
     {\empty} & {\empty} & \mathbb{E}^{m+1}_{\bar{s}}+v_{L} & {\subset} & \mathbb{E}^{m+3}_{\bar{s}+1} & \rotatebox{30}{$\subset$} & {\empty}  \label{nondeg_case}
  \end{array}
\end{equation}
\begin{equation}
  \begin{array}{ccccccc} 
     {\empty} & {\empty} & \Pi^{m+1}_{s,m-s,1}+v_{S} & {\subset} & \mathbb{E}^{m+3}_{s+1} & {\empty} & {\empty} \\
     {\empty} & \rotatebox{50}{$=$} &  {\empty} & {\empty} & {\empty} & \rotatebox{-50}{$\subset$} & {\empty} \\
     {\empty} & \rotatebox{15}{$=$} & \Pi^{m+1}_{s,m-s,1}+v_{T} & {\subset} & \mathbb{E}^{m+3}_{s+2} & \rotatebox{-30}{$\subset$} & {\empty} \\
     E^{\ast} & {\empty} & {\empty} & {\empty} & {\empty} & {\empty} & \mathbb{E}^{n+1}_{p}, \\
     {\empty} & \rotatebox{-15}{$=$} & \Pi^{m+1}_{s,m-s,1}+v_{L} & {\subset} & \mathbb{E}^{m+4}_{s+2} & \rotatebox{30}{$\subset$} & {\empty} \\
     {\empty} & \rotatebox{-50}{$=$} & {\empty} & {\empty} & {\empty} & \rotatebox{50}{$\subset$} & {\empty} \\
     {\empty} & {\empty} & \Pi^{m+1}_{s,m-s,1}+N & {\subset} & \mathbb{E}^{m+2}_{s+1} & {\empty} & {\empty} \label{deg_case}
  \end{array}
\end{equation}
where, by isometry of $\mathbb{S}^{n}_{p}(1)$, $v_{S}, v_{T}$ and $v_{L}$ are given by
\begin{align}
v_{S} &=  (\underbrace{0, \cdots, 0}_{\bar{s}}, 0, \cdots, 0, \rho) \in \mathbb{E}^{m+2}_{\bar{s}} \quad (\rho \geq 0), \label{First} \\
v_{T} &=  (\underbrace{\rho, 0, \cdots, 0}_{\bar{s}+1}, 0, \cdots, 0) \in \mathbb{E}^{m+2}_{\bar{s}+1} \quad (\rho \geq 0), \\
v_{L} &=  (\underbrace{1, 0, \cdots, 0}_{\bar{s}+1}, 0, \cdots, 0, 1) \in \mathbb{E}^{m+3}_{\bar{s}+1} 
\end{align}
in the non-degenerate case (\ref{nondeg_case}), 
\begin{align}
v_{S} &=  (0, \underbrace{0, \cdots, 0}_{s}, 0, \cdots, 0, \rho, 0) \in \mathbb{E}^{m+3}_{s+1} \quad (\rho \geq 0), \\
v_{T} &=  (0, \underbrace{\rho, 0, \cdots, 0}_{s+1}, 0, \cdots, 0, 0) \in \mathbb{E}^{m+3}_{s+2} \quad (\rho \geq 0), \\
v_{L} &=  (0, \underbrace{1, 0, \cdots, 0}_{s+1}, 0, \cdots, 0, 1, 0) \in \mathbb{E}^{m+4}_{s+2} \label{Last}
\end{align}
in the degenerate case (\ref{deg_case}). 
The proof is completed by checking the intersection of $\mathbb{S}^{n}_{p}(1)$ and $E^{\ast}$ in each case (\ref{First})--(\ref{Last}). \qed

\section{Observations and Applications}
\subsection{Observation $1$ : Riemannian or Lorentzian cases}
We first restore the classification of totally umbilical submanifolds in spheres and hyperbolic spaces, i.e. Riemannian case $p=0$.
\begin{itemize}
\item Totally umbilical submanifolds of $\mathbb{S}^{n}(1)$:
\begin{itemize}
\item[(1)] $\mathbb{S}^{m}(1) \rightarrow \mathbb{S}^{m+1}(1)$ ; $x \mapsto (x,0)$ \quad (totally geodesic);
\item[(2)] $\mathbb{S}^{m}\left(r^2\right) \rightarrow \mathbb{S}^{m+1}(1)$ ; $x \mapsto (x,\sqrt{1-r^2})$ \quad ($0<r<1$).
\end{itemize}
\item Totally umbilical submanifolds of $\mathbb{H}^{n}(-1)$: 
\begin{itemize}
\item[(1)] $\mathbb{H}^{m}(-1) \rightarrow \mathbb{H}^{m+1}(-1) \ ; \ x \mapsto (x,0)$ \quad (totally geodesic); 
\item[(2)] $\mathbb{H}^{m}\left(-r^2\right) \rightarrow \mathbb{H}^{m+1}(-1) \ ; \  x \mapsto (x, \sqrt{r^2-1}) \quad (r>1)$; 
\item[(3)] $\mathbb{S}^{m}\left(r^2\right) \rightarrow \mathbb{H}^{m+1}(-1) \ ; \  x \mapsto (\sqrt{1+r^2}, x) \quad (r>0)$; 
\item[(4)] $\mathbb{E}^{m} \rightarrow \mathbb{H}^{m+1}(-1) \ ; \  x \mapsto \left(||x||^{2}+\dfrac{5}{4},x,||x||^{2}+\dfrac{3}{4}\right)$. 
\end{itemize}
\end{itemize}
In de~Sitter and anti-de~Sitter spacetimes, i.e. Lorentzian case $p=1$, we obtain the followings:
\begin{itemize}
\item Totally umbilical submanifolds of $d\mathbb{S}^{n}(1)$: 
\begin{itemize}
\item[(1)] $d\mathbb{S}^{m}(1) \rightarrow d\mathbb{S}^{m+1}(1)$ ; $x \mapsto (x,0)$ \quad (totally geodesic); 
\item[(2)] $\mathbb{S}^{m}(1) \rightarrow d\mathbb{S}^{m+1}(1)$ ; $x \mapsto (0,x)$ \quad (totally geodesic); 
\item[(3)] $d\mathbb{S}^{m}\left(r^2\right) \rightarrow d\mathbb{S}^{m+1}(1)$ ; $x \mapsto (x,\sqrt{1-r^2})$ \quad ($0<r<1$); 
\item[(4)] $\mathbb{S}^{m}\left(r^2\right) \rightarrow d\mathbb{S}^{m+1}(1)$ ; $x \mapsto (\sqrt{r^2 - 1},x)$ \quad ($r>1$); 
\item[(5)] $\mathbb{S}^{m}(1) \rightarrow d\mathbb{S}^{m+2}(1)$ ; $x \mapsto (1,x,1)$; 
\item[(6)] $\mathbb{H}^{m}\left(-r^2\right) \rightarrow d\mathbb{S}^{m+1}(1)$ ; $x \mapsto (x,\sqrt{1+r^{2}})$ \quad ($r>0$); 
\item[(7)] $\mathbb{E}^{m} \rightarrow d\mathbb{S}^{m+1}(1) \ ; \ x \mapsto \left(||x||^{2}-\dfrac{3}{4},x,||x||^{2}-\dfrac{5}{4}\right)$. 
\end{itemize}
\item Totally umbilical submanifolds of $Ad\mathbb{S}^{n}(-1)$: 
\begin{itemize}

\item[(1)] $\mathbb{H}^{m}(-1) \rightarrow Ad\mathbb{S}^{m+1}(-1) \ ; \ x \mapsto (x,0)$ \quad (totally geodesic); 
\item[(2)] $Ad\mathbb{S}^{m}(-1) \rightarrow Ad\mathbb{S}^{m+1}(-1) \ ; \  x \mapsto (0,x)$ \quad (totally geodesic); 
\item[(3)] $\mathbb{H}^{m}\left(-r^2\right) \rightarrow Ad\mathbb{S}^{m+1}(-1) \ ; \  x \mapsto (\sqrt{1-r^2}, x) \quad (0<r<1)$; 
\item[(4)] $Ad\mathbb{S}^{m}(-r^2) \rightarrow Ad\mathbb{S}^{m+1}(-1) \ ; \  x \mapsto (x, \sqrt{r^2-1}) \quad (r>1)$; 
\item[(5)] $\mathbb{H}^{m}(-1) \rightarrow Ad\mathbb{S}^{m+2}(-1) \ ; \  x \mapsto (1,x,1)$; 
\item[(6)] $d\mathbb{S}^{m}\left(r^2\right) \rightarrow Ad\mathbb{S}^{m+1}(-1) \ ; \  x \mapsto (\sqrt{1+r^2}, x) \quad (r>0)$; 
\item[(7)] $\mathbb{L}^{m} \rightarrow Ad\mathbb{S}^{m+1}(-1) \ ; \  x \mapsto \left(\langle x, x \rangle_{1}+\dfrac{5}{4},x,\langle x, x \rangle_{1}+\dfrac{3}{4}\right)$. 
\end{itemize}
\end{itemize}
Remark that $||\cdot||$ is the canonical Euclidean norm of $\mathbb{E}^{m}$, i.e. $||x||^2:=\langle x, x \rangle_{0}$. 

\subsection{Observation $2$ : Totally umbilical lightlike submanifolds}
As a by-product of the proof of main results, 
we obtain the following lightlike submanifolds in non-flat space forms. 

\begin{prop} \rm \label{lightlike1}
Let $m \geq 2$. 
The followings are full totally umbilical lightlike submanifolds in a pseudo-sphere $\mathbb{S}^{n}_{p}(1)$. 
\begin{itemize}
\item[(1)] $\mathbb{S}^{m-1}_{s}(1) \times \mathbb{E}^{0,0,1} \rightarrow \mathbb{S}^{m+1}_{s+1}(1)$ ; $(x,t) \mapsto (t,x,t)$ \quad (totally geodesic),
\item[(2)] $\mathbb{S}^{m-1}_{s}(r^2) \times \mathbb{E}^{0,0,1} \rightarrow \mathbb{S}^{m+2}_{s+1}(1)$ ; $(x,t) \mapsto (t,x,\sqrt{1-r^{2}},t)$ \quad ($0<r<1$),
\item[(3)] $\mathbb{S}^{m-1}_{s}(r^2) \times \mathbb{E}^{0,0,1} \rightarrow \mathbb{S}^{m+2}_{s+2}(1)$ ; $(x,t) \mapsto (t,\sqrt{r^{2}-1},x,t)$ \quad ($r>1$),

\item[(4)] $\mathbb{H}^{m-1}_{s}(-r^2) \times \mathbb{E}^{0,0,1} \rightarrow \mathbb{S}^{m+2}_{s+1}(1)$ ; $(x,t) \mapsto (t,x,\sqrt{1+r^{2}},t)$ \quad ($r>0$),

\item[(5)] $\Lambda^{m}_{s} \rightarrow \mathbb{S}^{m+1}_{s+1}(1)$ ; $x \mapsto (x,1)$,
\item[(6)] $\Lambda^{m}_{s} \times \mathbb{E}^{0,0,1} \rightarrow \mathbb{S}^{m+2}_{s+2}(1)$ ; $(x,t) \mapsto (t,x,1,t)$,
\item[(7)] $\mathbb{S}^{m-1}_{s}(1) \times \mathbb{E}^{0,0,1} \rightarrow \mathbb{S}^{m+3}_{s+2}(1)$ ; $(x,t) \mapsto (t,1,x,1,t)$.
\end{itemize}
We remark that the above (6) is $2$-lightlike, and others are $1$-lightlike. 
\end{prop} 

\begin{prop} \rm \label{lightlike2}
Let $m \geq 2$. 
The followings are full totally umbilical lightlike submanifolds in a pseudo-hyperbolic space $\mathbb{H}^{n}_{p}(-1)$.  
\begin{itemize}
\item[(1)] $\mathbb{H}^{m-1}_{s}(-1) \times \mathbb{E}^{0,0,1} \rightarrow \mathbb{H}^{m+1}_{s+1}(-1)$ ; $(x,t) \mapsto (t,x,t)$ \quad (totally geodesic),
\item[(2)] $\mathbb{H}^{m-1}_{s}(-r^2) \times \mathbb{E}^{0,0,1} \rightarrow \mathbb{H}^{m+2}_{s+1}(-1)$ ; $(x,t) \mapsto (t,\sqrt{1-r^{2}},x,t)$ \quad ($0<r<1$), 
\item[(3)] $\mathbb{H}^{m-1}_{s}(-r^2) \times \mathbb{E}^{0,0,1} \rightarrow \mathbb{H}^{m+2}_{s+2}(-1)$ ; $(x,t) \mapsto (t,x,\sqrt{r^{2}-1},t)$ \quad ($r>1$), 

\item[(4)] $\mathbb{S}^{m-1}_{s}(r^2) \times \mathbb{E}^{0,0,1} \rightarrow \mathbb{H}^{m+2}_{s+1}(-1)$ ; $(x,t) \mapsto (t,\sqrt{1+r^{2}},x,t)$ \quad ($r>0$),

\item[(5)] $\Lambda^{m}_{s} \rightarrow \mathbb{H}^{m+1}_{s+1}(-1)$ ; $x \mapsto (1,x)$,
\item[(6)] $\Lambda^{m}_{s} \times \mathbb{E}^{0,0,1} \rightarrow \mathbb{H}^{m+2}_{s+2}(-1)$ ; $(x,t) \mapsto (t,1,x,t)$,
\item[(7)] $\mathbb{H}^{m-1}_{s}(-1) \times \mathbb{E}^{0,0,1} \rightarrow \mathbb{H}^{m+3}_{s+2}(-1)$ ; $(x,t) \mapsto (t,1,x,1,t)$.
\end{itemize}
We remark that the above (6) is $2$-lightlike, and others are $1$-lightlike. 
\end{prop} 
A classification problem of totally umbilical lightlike submanifolds in pseudo-Riemannin space forms is open. 
It is hard to classify totally umbilical submanifolds in pseudo-Riemannian space forms. 
We state one of the evidences below. 
Let $s$ be a non-negative integers. 
We consider the following example. 
\begin{align*}
E^{\ast} &:= \Pi^{m+2}_{s,m-s,2} + N \subset \mathbb{E}^{m+4}_{s+2} \\
&= \left\{\left(w - \frac{1}{2}, t, u_{1}, \cdots, u_{s}, v_{1}, \cdots, v_{t}, t, w + \frac{1}{2}\right) \ \middle| \ t, u_{i}, v_{j}, w \in \mathbb{R} \right\}, 
\end{align*}
where $N = \frac{1}{2}(-1, 0, \cdots, 0, 1) \in \mathbb{E}^{m+4}_{s+2}$. 
Then, we see that 
\begin{align*}
\mathcal{S}^{m+1} := 
\mathbb{S}^{m+3}_{s+2}(1) \cap E^{\ast} &= \left\{\left(w - \frac{1}{2}, x, w + \frac{1}{2}\right) \ \middle| \ -\left(w-\frac{1}{2}\right)^{2} + \langle x, x\rangle_{s+1} + \left(w+\frac{1}{2}\right)^{2} = 1 \right\} \\
&= \left\{\left(-\frac{1}{2}\langle x, x\rangle_{s+1}, x, 1 - \frac{1}{2}\langle x, x\rangle_{s+1} \right) \ \middle| \ x \in \Pi^{m+1}_{s,m-s,1} \right\}, 
\end{align*}
where we set $x = (t, u_{1}, \cdots, u_{s}, v_{1}, \cdots, v_{t}, t) \in \Pi^{m+1}_{s,m-s,1} \subset \mathbb{E}^{m+2}_{s+1}$. 
Moreover, it holds 
\begin{align*}
\mathcal{S}^{m+1} &\subset \mathcal{H}^{m+2} := \left\{\left(-\frac{1}{2}\langle \tilde{x}, \tilde{x} \rangle_{s+1}, \tilde{x}, 1 - \frac{1}{2}\langle \tilde{x}, \tilde{x} \rangle_{s+1} \right) \ \middle| \ \tilde{x} \in \mathbb{E}^{m+2}_{s+1} \right\} \cong \mathbb{E}^{m+2}_{s+1} \\
&\subset \mathbb{S}^{m+3}_{s+2}(1), 
\end{align*}
where we set $\tilde{x} = (u_{0}, u_{1}, \cdots, u_{s}, v_{1}, \cdots, v_{t}, v_{0}) \in \mathbb{E}^{m+2}_{s+1}$. 
Since $\Pi^{m+1}_{s,m-s,1}$ is a flat totally geodesic $1$-lightlike hypersurface in $\mathbb{E}^{m+2}_{s+1}$, and $\mathcal{H}^{m+2}$ is a flat totally umbilical hypersurface in $\mathbb{S}^{m+2}_{s+2}(1)$, we can claim that $\mathcal{S}^{m+1}$ is a flat totally umbilical $2$-lightlike submanifold in $\mathbb{S}^{m+3}_{s+2}(1)$. We remark that $\mathcal{H}^{m+2}$ is congruent to (7) in Theorem~\ref{main1}. 
The submanifold $\mathcal{S}^{m+1}$ is an example which cannot be obtained by the proof of Theorem \ref{main1} and \ref{main2}. 
If we consider the higher co-dimensional case, 
we can construct a $1$-parameter family of flat totally umbilical lightlike submanifolds 
\begin{align*}
&\mathbb{E}^{0,m,1} \rightarrow \mathbb{S}^{m+3}_{2}(1) \\ 
& \ \ \vin \qquad \qquad \vin \\
&(x, r) \mapsto \biggl(-\frac{1}{2}||x||^{2}\cos{\theta}-(r-\theta)\sin{\theta}, -\frac{1}{2}||x||^{2}\sin{\theta} + (r-\theta)\cos{\theta}, x  \\
& \qquad \qquad \qquad \frac{1}{2}(2 - ||x||^{2})\sin{\theta}+(r-\theta)\cos{\theta}, \frac{1}{2}(2 - ||x||^{2})\cos{\theta}-(r-\theta)\sin{\theta} \biggr),
\end{align*}
where $\theta \in \mathbb{R}$ and we set $(x, r) \in \mathbb{E}^{0,m,1} =  \mathbb{E}^{m} \oplus \mathbb{E}^{0,0,1}$. 
When $\theta = 0$, we obtain the case of $\mathcal{S}^{m+1}$. 
Therefore, it is expected that there are many other examples which are given by the intersection a pseudo-sphere and an affine subspace and a higher dimension in a pseudo-Euclidean space. 
On the other hand, the following result is known: 

\begin{prop}[{\cite[Proposition~5.3, Chapter~4]{BD}}] \rm
Any lightlike surface $M^{2}$ of a three-dimensional Lorentzian manifold is either totally
umbilical or totally geodesic. 
\end{prop} 

This is the reason why a classification of totally umbilical lightlike submanifolds is much more complicated than that of totally umbilical non-degenerate submanifolds. 

We will observe co-dimension two and co-index one totally umbilical submanifolds in Theorem~\ref{tot_umb_Euc}, \ref{main1} and \ref{main2}. When we define a hyperplane in 
$\mathbb{E}^{m+2}_{s+1}$ by 
\begin{equation*} 
N^{m+1}(0) := \left\{\left(t+\frac{1}{4},x,t-\frac{1}{4}\right) \in \mathbb{E}^{m+2}_{s+1} \mid t \in \mathbb{R}, x \in \mathbb{E}^{m}_{s} \right\}. 
\end{equation*}
This is a totally geodesic $1$-lightlike hypersurface in $\mathbb{E}^{m+2}_{s+1}$. 
We can regard a flat marginally trapped submanifold $\mathbb{U}^{m}_{s}$ as a hypersurface in $N^{m+1}(0)$ by 
\begin{equation} 
\mathbb{E}^{m}_{s} \ni x \mapsto \left(\langle x, x \rangle_{s} + \frac{1}{4}, x, \langle x, x \rangle_{s} - \frac{1}{4}\right) \in N^{m+1}(0) \subset \mathbb{E}^{m+2
}_{s+1}. \label{analogy}
\end{equation}
In addition, we can also regard as a hypersurface in the lightcone $\Lambda^{m+1}_{s}$ by
\begin{equation*} 
\mathbb{E}^{m}_{s} \ni x \mapsto \left(\langle x, x \rangle_{s} + \frac{1}{4}, x, \langle x, x \rangle_{s} - \frac{1}{4}\right) \in \Lambda^{m+1}_{s} \subset \mathbb{E}^{m+2
}_{s+1}. 
\end{equation*}

Here set $\varepsilon = \pm 1$. A hypersurface in $\mathbb{M}^{m+2}_{s+1}(\varepsilon)$ defined by 
\begin{equation*} 
N^{m+1}(\varepsilon) := \{(t,x,t) \in \mathbb{M}^{m+2}_{s+1}(\varepsilon) \ | \ t \in \mathbb{R}, x \in \mathbb{M}^{m}_{s}(\varepsilon) \} 
\end{equation*}
is a totally geodesic $1$-lightlike hypersurface in $\mathbb{M}^{m+2}_{s+1}(\varepsilon)$. 
For the co-dimension two and co-index one totally umbilical isometric embedding in given (\ref{missing}), say $\psi$, in Theorem \ref{main1} and \ref{main2}, we can regard $\psi$ as a hypersurface in $N^{m+1}(1)$, i.e. 
\begin{equation*} 
\psi : \mathbb{M}^{m}_{s}(\varepsilon) \ni x \mapsto (1,x,1) \in N^{m+1}(\varepsilon) \subset \mathbb{M}^{m+2}_{s+1}(\varepsilon). 
\end{equation*}
This is an analogue of the consideration of the mapping (\ref{analogy}). 

From Proposition~\ref{lightlike1} and \ref{lightlike2}, isometric embeddings of $\Lambda^{m+1}_{s}$ into $\mathbb{M}^{m+2}_{s+1}(\varepsilon)$
\begin{equation*} 
\chi : \Lambda^{m+1}_{s} \rightarrow \mathbb{M}^{m+2}_{s+1}(\varepsilon) \ ; \  
  \begin{cases}
   x \mapsto (x,1) \quad (\varepsilon = 1), \\
   x \mapsto (1,x) \quad (\varepsilon = -1)
  \end{cases}
\end{equation*}
are totally umbilical $1$-lightlike hypersurfaces. 
On the other hand, non-flat space forms $\mathbb{M}^{m}_{s}(\varepsilon)$ are isometrically embedded in the lightcone $\Lambda^{m+1}_{s}$ as follows 
\begin{equation*} 
\rho : \mathbb{M}^{m}_{s}(\varepsilon) \rightarrow \Lambda^{m+1}_{s} \ ; \ 
  \begin{cases}
   x \mapsto (1, x) \quad (\varepsilon = 1), \\
   x \mapsto (x, 1) \quad (\varepsilon = -1).
  \end{cases}
\end{equation*}
For $\psi$, we can see there exists a nested structure of space forms via the lightcone 
\begin{equation*} 
\psi = \chi \circ \rho : \mathbb{M}^{m}_{s}(\varepsilon) \overset{\rho}{\hookrightarrow} \Lambda^{m+1}_{s} \overset{\chi}{\hookrightarrow} \mathbb{M}^{m+2}_{s+1}(\varepsilon) \ ; \ x \mapsto (1,x,1). 
\end{equation*} 
In summary, we can find out the following relation among pseudo-Riemannian space forms and lightcones: 

\begin{figure}[H]
\begin{center}
\includegraphics[scale=0.85]{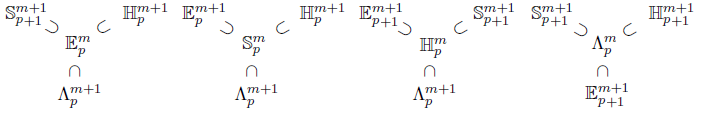}
  \caption{Totally umbilical inclusion relations.}
\end{center}
\end{figure}

\subsection{Application $1$ : The moduli space of isometric immersions}
Let $M^{m}_{s}, \bar{M}^{n}_{p}$ be pseudo-Riemannian manifolds, and $g, \bar{g}$ their pseudo-Riemannian metrics, respectively. 
We define a mapping space 
\begin{equation} 
\{ \phi \in C^{\infty}\left(M^{m}_{s}, \bar{M}^{n}_{p}\right) \mid \phi^{\ast}\bar{g} = g \}, \label{map_sp}
\end{equation}
where $C^{\infty}\left(M^{m}_{s}, \bar{M}^{n}_{p}\right)$ denotes the set of all smooth mapping from $M^{m}_{s}$ into $\bar{M}^{n}_{p}$. 
We introduce the compact open $C^{\infty}$-topology in the set and consider it as a topological space. 
The isometric group of $\bar{M}^{n}_{p}$ naturally acts on this mapping space (\ref{map_sp}). 
We call the quotient space by the action the \textit{moduli space} for isometric immersions $\phi : M^{m}_{s} \rightarrow \bar{M}^{n}_{p}$, denoted by $\mathcal{M}(M^{m}_{s},\bar{M}^{n}_{p})$. 
We shall denote the moduli space for totally umbilical isometric immersions by $\mathcal{M}_{umb}(M^{m}_{s},\bar{M}^{n}_{p})$. 
It is obvious that $\mathcal{M}_{umb}(M^{m}_{s},\bar{M}^{n}_{p}) \subset \mathcal{M}(M^{m}_{s},\bar{M}^{n}_{p})$ as a subspace. 

\begin{prop} \rm
If $\varepsilon = 0, \pm1$, then it holds 
\begin{equation*} 
\mathcal{M}_{umb}\left(\mathbb{M}^{m}_{s}(\varepsilon),\mathbb{M}^{n}_{p}(\varepsilon)\right) \overset{\textrm{homeo}}{\cong}
  \begin{cases}
   \{\textrm{$\ast$}\} \quad \left(n = m+1, p = s , s+1,\textrm{or} \ n = m+2, p=s\right), \\
   (X, \mathcal{O}_{X}) \quad \left(n \geq m+2, \ p \geq s + 1\right),
  \end{cases}
\end{equation*}
where $\{\textrm{$\ast$}\}$ is the one-point space, and a topological space $(X, \mathcal{O}_{X})$ is defined by 
\begin{equation*} 
X := \{g, u\}, \quad \mathcal{O}_{X} := \{ \emptyset, \{u\}, X \}. 
\end{equation*}
Here the elements $g$ and $u$ of $X$ express the congruent classes of totally geodesic isometric immersions and non-totally geodesic, totally umbilical isometric immersions, respectively. 
Moreover, the space $(X, \mathcal{O}_{X})$ is connected, non-Hausdorff. 
\end{prop}

\begin{proof}
From Theorem \ref{tot_umb_Euc}, \ref{main1} and \ref{main2}, the set $\mathcal{M}_{umb}\left(\mathbb{M}^{m}_{s}(\varepsilon),\mathbb{M}^{n}_{p}(\varepsilon)\right)$ is a one-point set when $n = m+1, p = s , s+1$ or $n = m+2, p=s$, and a two-point set when $n \geq m+2, \ p \geq s + 1$. 
In the case $n \geq m+2, \ p \geq s + 1$, 
we consider the following totally umbilical isometric immersions, for each $a \in \mathbb{R}$, 
\begin{align*} 
\psi_{a} &: \mathbb{M}^{m}_{s}(\varepsilon) \rightarrow \mathbb{M}^{m+2}_{s+1}(\varepsilon) \ ; \ x \mapsto (a,x,a) \quad (\varepsilon=\pm 1), \\
\psi_{a} &: \mathbb{M}^{m}_{s}(\varepsilon) \rightarrow \mathbb{M}^{m+2}_{s+1}(\varepsilon) \ ; \ x \mapsto (a\langle x,x \rangle_{s}, x , a\langle x, x \rangle_{s}) \quad (\varepsilon=0). 
\end{align*}
They are congruent to $\psi_{1}$ if $a \neq 0$, and a totally geodesic isometric immersion $\psi_{0}$ if $a = 0$. On the other hand, it is obvious that 
\begin{equation*} 
\lim_{a\rightarrow0} \psi_{a} = \psi_{0}. 
\end{equation*}
Therefore, $\mathcal{M}_{umb}\left(\mathbb{M}^{m}_{s}(\varepsilon),\mathbb{M}^{n}_{p}(\varepsilon)\right)$ is not a discrete space. Since mean curvature vector fields $H_{0}, H_{1}$ of $\psi_{0}, \psi_{1}$ entirely satisfy $H_{0}=0$ (closed condition) and $H_{1} \neq 0$ (open condition), respectively, we obtain the conclusion. 
\end{proof}

\begin{cor} \rm \label{mod_isom_imm}
Let $n \geq m+2, \ p \geq s + 1, \ \varepsilon = 0, \pm1$, then the moduli space of isometric immerisons between space forms which are of the same constant curvature
\begin{equation*} 
\mathcal{M}\left(\mathbb{M}^{m}_{s}(\varepsilon),\mathbb{M}^{n}_{p}(\varepsilon)\right) 
\end{equation*}
is a non-Hausdorff space.
\end{cor}

\subsection{Application $2$ : Parallel submanifolds}
An isometric immersion $\phi : M^{m}_{s} \rightarrow \mathbb{M}^{n}_{p}(\varepsilon)$ is called \textit{substantial} if it is not contained in any non-degenerate totally umbilical submanifold in $\mathbb{M}^{n}_{p}(\varepsilon)$. 
By definition, if $\phi$ is substantial, then $\phi$ is full. 

Here, we define two additional classes of submanifolds. 
Let $M$ and $\bar{M}$ be pseudo-Riemannian manifolds, and $\phi : M \rightarrow \bar{M}$ an isometric immersion. 
We call the immersion $\phi$ \textit{parallel} if the covariant derivative of the second fundamental form $h$ of $\phi$ vanishes identically, i.e. for any $X, Y, Z \in \Gamma(TM)$,
\begin{equation*} 
(\tilde{\nabla}_{X}h)(Y,Z) := \nabla^{\bot}_{X}h(Y,Z) - h(\nabla_{X}Y,Z) - h(Y,\nabla_{X}Z) = 0, 
\end{equation*}
where $\nabla, \nabla^{\bot}$ are the Levi-Civita and the normal connections, respectively. 
We call the immersion $\phi$ \textit{symmetric} if, for each $x \in M$, there exist isometries $\Phi_{x} \in \textrm{Isom}(M), \Psi_{x} \in \textrm{Isom}(\bar{M})$ such that 
\begin{equation*} 
\Phi_{x}(x)=x, \quad \Psi_{x} \circ \phi = \phi \circ \Phi_{x}, \quad (\Psi_{x})_{\ast}\phi_{\ast}X=-\phi_{\ast}X, \quad (\Psi_{x})_{\ast}(\xi)=\xi, 
\end{equation*}
where $X \in T_{x}M$ and $\xi \in T^{\bot}_{x}M$.
We may consider the local version of the above conditions. 
Namely, for each $x \in M$, $\Phi_{x}$ and $\Psi_{x}$ are local isometries, 
we call $\phi$ \textit{locally symmetric}. 
When an isometric immersion $\phi : M \rightarrow \bar{M}$ is parallel or (locally) symmetric, 
we call the image $\phi(M)$ a parallel or (locally) symmetric submanifold in $\bar{M}$, respectively. 

When an ambient space $\bar{M}$ is a pseudo-Riemannian space form, a submanifold $M$ of $\bar{M}$ is parallel if and only if it is locally symmetric. Moreover, $M$ is complete parallel if and only if it is symmetric \cite{B}. 

\begin{lem} \rm
Let $M'$ and $\bar{M}$ be pseudo-Riemannian space forms, and let $\phi : M \rightarrow M'$ and $\psi : M' \rightarrow \bar{M}$ isometric immersions. 
We assume that $\psi$ is totally umbilical, that is, $M'$ is embedded as a totally umbilical submanifold in $\bar{M}$. 
Then, $\phi$ is parallel if and only if $\psi \circ \phi$ is parallel. 
\end{lem}

\begin{proof}
In Riemannian case, refer \cite[Lemma~3.7.5]{BCO}. 
The proof is done by the same argument since totally umbilical submanifolds in pseudo-Riemannian space forms are of parallel mean curvature vector fields from Lemma \ref{umb_umb}. 
\end{proof}

\begin{prop} \rm \label{construction}
Let $\phi : M^{m}_{s} \rightarrow \mathbb{S}^{n}_{p}(1)$ be a substantial isometric immersion, and $\psi : \mathbb{S}^{n}_{p}(1) \rightarrow \mathbb{S}^{n+2}_{p+1}(1)$ a totally umbilical isometric immersion defined by
\begin{equation} 
\psi(x) = (1, x, 1) \in \mathbb{S}^{n+2}_{p+1}(1) \quad (x \in \mathbb{S}^{n}_{p}(1)). \label{codim_two}
\end{equation}
Then, the composition $\psi \circ \phi : M^{m}_{s} \rightarrow \mathbb{S}^{n+2}_{p+1}(1)$ is full parallel isometric immersion. 
Moreover, let $\iota$ be the totally geodesic inclusion 
\begin{equation*} 
\iota(x) = (0,x,0) \in \mathbb{S}^{n+2}_{p+1}(1) \quad (x \in \mathbb{S}^{n}_{p}(1)). 
\end{equation*}
Then, $\psi \circ \phi$ is not congruent to $\iota \circ \phi$ in $\mathbb{S}^{n+2}_{p+1}(1)$.
\end{prop}

\begin{rem} \rm
Proposition \ref{construction} is valid when the ambient space is a pseudo-hyperbolic space $\mathbb{H}^{n}_{p}(-1)$. 
In case of Riemannian parallel surfaces, this construction is known in \cite[(C) of Theorem~9.1 and  Theorem~10.1]{C1}. 
\end{rem}

Let $G$ and $K$ be a Lie group and its closed Lie subgroup, respectively, and $G/K$ an irreducible indefinite symmetric $R$-space such as indefinite Grassmann manifolds, indefinite orthogonal groups and complex spheres etc \cite{B, Na}. 
Let $f : G/K \rightarrow \mathbb{E}^{n+1}_{p}$ be a standard embedding. Then, by scaling of the metric, $f(G/K)$ is a minimal submanifold of $\mathbb{S}^{n}_{p}(1)$ or $\mathbb{H}^{n}_{p-1}(-1)$. 
When $\psi$ is the co-dimension two and co-index one totally umbilical isometric embedding (\ref{codim_two}), considering the composition $\psi \circ f$, we obtain a full complete parallel isometric embedding in $\mathbb{S}^{n+2}_{p+1}(1)$ or $\mathbb{H}^{n+2}_{p}(-1)$ by using Proposition~{\ref{construction}}. However, its mean curvature vector field $H$ of $\psi \circ f$ is non-zero and satisfies $\langle H,H\rangle_{p+1} = 0$. Namely, $\psi \circ f$ is a marginally trapped isometric immersion. 

On the other hand, a full parallel, minimal isometric immersion of an irreducible Riemannian symmetric $R$-space into a unit sphere is rigid, i.e. it is congruent to a standard embedding. However, in the indefinite case, there exist full parallel, marginally trapped isometric immersions of irreducible indefinite symmetric $R$-spaces into unit pseudo-spheres which are not congruent to standard embeddings. See also Blomstrom's rigidity theorem \cite[Theorem~3]{B}. 

B.~Y.~Chen et al. classified Riemannian and Lorentzian parallel surfaces in pseudo-Riemannian space forms. In \cite{C}, he commented that the explicit classifications of parallel submanifolds in pseudo-Riemannian space forms are much more complicated than that of Riemannian situations. In fact, it is known that there exist $24$ families and $53$ families of parallel Lorentzian surfaces in neutral space forms $\mathbb{S}^{4}_{2}(1)$ and $\mathbb{H}^{4}_{2}(-1)$, respectively. 
Some of these surfaces are full but not substantial. 
Regarding Riemannian parallel surfaces in $\mathbb{S}^{n}_{p}(1)$, we see in \cite{CV} the following flat complete parallel surfaces 
\begin{equation*} 
f : \mathbb{E}^{2} \ni (u,v) \mapsto \left(v^{2}+a^{2}-\frac{3}{4}, a\cos{u}, a\sin{u}, v, v^{2}+a^{2}-\frac{5}{4} \right) \ (a>0). 
\end{equation*}
This parallel surface is full but not substantial. In fact, we set 
\begin{align}
\phi &: \mathbb{E}^{3} \ni (x,y,z) \mapsto \left(x^{2}+y^{2}+z^{2}-\frac{3}{4}, x, y, z, x^{2}+y^{2}+z^{2}-\frac{5}{4}\right) \in \mathbb{S}^{4}_{1}(1), \label{hypersurf} \\
\psi &: \mathbb{E}^{2} \ni (u,v) \mapsto (a\cos{u}, a\sin{u}, v) \in \mathbb{E}^{3} \ (a>0). \nonumber
\end{align}
Then, by direct calculation, we see $\phi \circ \psi = f$. Since the hypersurface (\ref{hypersurf}) is totally umbilical, $f$ is full but not substantial. 
Via totally umbilical isometric immersions 
\begin{align*}
&\mathbb{E}^{n}_{p} \ni x \mapsto \left(\langle x, x \rangle_{s}-\dfrac{3}{4},x,\langle x, x \rangle_{s}-\dfrac{5}{4}\right) \in \mathbb{S}^{n+1}_{p+1}(1), \\
&\mathbb{S}^{n}_{p}(1) \ni x \mapsto (1,x,1) \in \mathbb{S}^{n+2}_{p+1}(1),
\end{align*}
substantial parallel submanifolds in $\mathbb{E}^{n}_{p}$ or $\mathbb{S}^{n}_{p}(1)$ induce full parallel ones in $\mathbb{S}^{n+1}_{p+1}(1)$ and $\mathbb{S}^{n+2}_{p+1}(1)$, respectively. A classification of full parallel submanifolds may be difficult, but a classification of substantial complete ones may be possible. As further references, see also \cite{Na, B, Ka, Ka1, Ka2}.

\section*{Acknowledgment}
The author would like to express his deepest gratitude to his advisor, Professor
Takashi Sakai for valuable comments and advices. 
The author is also very grateful to Luiz Carlos Barbosa da Silva for his useful comments.

\end{document}